%% file: SPS_general_uniqueness_v12.tex
\documentclass[12pt]{article}

\usepackage{amssymb}
\usepackage{amsthm}
\usepackage{amsmath}
\usepackage{graphicx}
\usepackage{fullpage}
\usepackage{subcaption}
\usepackage{color}
\usepackage{enumerate}
 \numberwithin{equation}{section}
\usepackage{booktabs}
\allowdisplaybreaks

\usepackage[pdftex]{hyperref}
 \usepackage{mathtools}
 \mathtoolsset{showonlyrefs}
 \usepackage[nocompress]{cite}
 \usepackage{tikz}
\usetikzlibrary{arrows.meta}

\theoremstyle{plain}
\newtheorem{thm}{Theorem}[section]
\newtheorem{cor}[thm]{Corollary}

\newtheorem{lem}[thm]{Lemma}
\newtheorem{prop}[thm]{Proposition}

\theoremstyle{definition}
\newtheorem{defn}[thm]{Definition}

\theoremstyle{remark}

\newcommand{\N}{\mathbb{N}}
\newcommand{\R}{\mathbb{R}}

\newcommand{\E}{\mathbb{E}}

\newcommand{\bp}{\begin{proof}[\ensuremath{\mathbf{Proof}}]}
\newcommand{\bs}{\begin{proof}[\ensuremath{\mathbf{Solution}}]}
\newcommand{\ep}{\end{proof}}
\newcommand{\be}{\begin{equation}}
\newcommand{\ee}{\end{equation}}

\newcommand{\id}{\text{id}_{\R}}

\DeclareMathOperator*{\esslim}{ess\,lim}

\begin{document}

\title{Uniqueness of pressureless flow}

\author{Ryan Hynd\footnote{Department of Mathematics, University of Pennsylvania.  Partially supported by NSF grant DMS-2350454.}\; and Hwa Kil Kim\footnote{Department of Mathematics Education, Hannam University. Partially supported by NRF grant  RS$\cdot$ 2026$\cdot $25598735.}}

\maketitle

\begin{abstract}
We consider the pressureless Euler equations in one spatial dimension. These equations model the dynamics of particles on a line that interact only through perfectly inelastic collisions. The flow was first shown to be unique for given initial conditions by Huang and Wang, assuming the entropy and the strong initial continuity of energy conditions hold.  In this paper, we prove an analogous uniqueness theorem using a Lagrangian interpretation.  The key ingredients are showing that each solution satisfying the entropy condition and an initial kinetic energy bound has this particular type of Lagrangian interpretation and a monotonicity inequality involving conditional expectation. 
\end{abstract}

\section{Introduction}
The pressureless Euler equations in one spatial dimension are
\be\label{SPS}
\begin{cases}
\hspace{.26in}\partial_t\rho +\partial_x(\rho v)=0\\
\partial_t(\rho v) +\partial_x(\rho v^2)=0.
\end{cases}
\ee
These equations hold in $\R\times (0,\infty)$ and the unknowns are the mass distribution $\rho$ and the velocity $v$.
This system governs the dynamics of collections of particles on the real line which do not interact pairwise. These particles also may collide, and they undergo perfectly inelastic collisions when they do.

\par An analogous system of equations was considered in three dimensions by Zel'dovich in his work on the expansion of matter without pressure \cite{Zeldovich}. This was important in his theory for  early galaxy formation. Zel'dovich was able to establish short time existence of solutions for given initial conditions essentially by the method of characteristics \cite{Gurbatov}; see \eqref{method-of-char} below for a solution obtained by this method. When characteristics cross this corresponds to particle collisions. Unfortunately,  not much more is known in the higher dimensional setting; specifically, it is not known how to continue solutions past collisions for general initial conditions. 

\par The pressureless Euler equations also arise in the study of optimal mass transportation.  See Chapter 5 of \cite{MR1964483} and Chapter 13 of \cite{MR2459454}. In particular, this system of equations is satisfied by geodesics in the space of Borel probability measures on $\R$ when endowed with the quadratic Wasserstein distance.  However, this solution has its initial and terminal mass distribution prescribed rather than having its initial mass distribution and initial velocity prescribed. It corresponds to the case of a solution on a given time horizon in which particles never collide. 
\\
 \par {\bf Initial value problem}.  A basic problem is to verify that there is a solution for given initial conditions 
\be\label{IC}
\rho|_{t=0}=\rho_0 \quad\text{and}\quad v|_{t=0}=v_0.
\ee
One typically assumes $\rho_0$ is a finite Borel measure on $\R$ so that the particles initially have finite total mass. Without loss of generality, we may assume that the total mass is equal to one and consider $\rho_0$ as a Borel probability measure on $\R$. It is also natural to suppose that the kinetic energy is initially finite. That is, $v_0: \R\to \R$ is $\rho_0$ measurable with 
\be\label{Initialkinetic}
\int_{\R}v_0^2d\rho_0<\infty.
\ee

\par Let us recall the definition of a weak solution pair $\rho$ and $v$ of the pressureless Euler equations \eqref{SPS} which satisfies the initial conditions \eqref{IC}.  We will write ${\cal P}(\R)$ for the space of Borel probability measures on $\R$ and say $(\mu_k)_{k\in \N}$ in ${\cal P}(\R)$ converges  {\it narrowly} to $\mu\in {\cal P}(\R)$ if 
$$
\lim_{k\to\infty}\int_{\R}fd\mu_k=\int_{\R}fd\mu
$$
for each $f\in C_b(\R)$. See section 5.1 of \cite{AmbrosioGigliSavare2008} for more on this notion of convergence. 
\begin{defn}
Suppose $\rho_0\in {\cal P}(\R)$ and $v_0\in L^2(\rho_0)$. A narrowly continuous $\rho: [0,\infty)\rightarrow {\cal P}(\R); t\mapsto \rho_t$ and Borel $v:\R\times(0,\infty)\rightarrow\R$ is a {\it weak solution pair} of the pressureless Euler equations \eqref{SPS}
which satisfies the initial conditions \eqref{IC} provided
\be\label{velocityIntegrability}
\int^T_0\int_{\R}v^2d\rho_tdt<\infty
\ee
for each $T>0$,
\be\label{EP1weak}
\int^\infty_0\int_{\R}(\partial_t\phi+v\partial_x\phi)d\rho_tdt+\int_{\R}\phi(\cdot,0)d\rho_0=0,
\ee
and 
\be\label{EP2weak}
\int^\infty_0\int_{\R}(v\partial_t\phi +v^2\partial_x\phi)d\rho_tdt+\int_{\R}\phi(\cdot,0)v_0d\rho_0=0
\ee
for each $\phi\in C^1_c(\R\times[0,\infty))$.
\end{defn}

\par There are two classes of initial conditions in which it is possible to easily design a weak solution pair. The first is for any $\rho_0\in {\cal P}(\R)$ and $v_0: \R\to \R$ continuous and nondecreasing with $v_0\in L^2(\rho_0)$. In this case, a particle starting out at a position $x$ in the support of $\rho_0$ is at position $x+tv_0(x)$ at time $t$. In particular, the distance between any two particles is nondecreasing and particles never collide.   Moreover, it is easily verified that 
\be\label{method-of-char}
\rho_t=(\id+tv_0)_{\#}\rho_0\quad\text{and}\quad v(\cdot,t)=v_0\circ (\id+tv_0)^{-1}
\ee
specifies a weak solution pair. Here $T_{\#}\mu(A)=\mu(T^{-1}(A))$ for Borel $T:\R\to\R$, Borel $A\subset \R$, and $\mu\in {\cal P}(\R)$. 
\input{noncolliding-figure.tex}

\par Another example is when the initial mass distribution is a convex combination of point masses $\rho_0=\sum^N_{i=1}m_i\delta_{x_i}.$   In this setting, particles move in straight lines unless they collide and undergo a perfectly inelastic collision. In such an event, the masses join and the instantaneous velocity of the larger mass is the mass average of the velocities of the individual particles immediately prior to the collision.  If $\gamma_i:[0,\infty)\to \R$ denotes the piecewise linear trajectory of the particle initially located at $x_i$, then it turns out that
\be\label{eq:sticky-soln-ex}
\rho_t=\sum^N_{i=1}m_i\delta_{\gamma_i(t)}
\quad \text{and}\quad
v(x,t)=
\begin{cases}
\dot\gamma_i(t+), \quad & x=\gamma_i(t)\\
0, \quad &\text{otherwise}
\end{cases}
\ee
is a weak solution pair \cite{MR4007615}.
\input{colliding-schematic.tex}

\par For general initial conditions $\rho_0$ and $v_0$, a weak solution pair was verified initially by E, Rykov, and Sinai \cite{ERykovSinai1996} and by Brenier and Grenier \cite{BrenierGrenier1998}. There have  also been many subsequent works on existence including \cite{NguyenTudorascu2008, MR4007615,NguyenTudorascu2015, Dermoune1999StickyProbab,MR2540269,Hynd2020AnnaliSticky, CavallettiSedjroWestdickenberg2015}. 
\\
\par {\bf Two conditions}.  Weak solution pairs of the initial value problem above can be constructed to have two additional properties.  The first is the {\it everywhere entropy condition}
\be\label{Eentropy}
(v(x,t)-v(y,t))(x-y)\le\frac{1}{t}(x-y)^2\; \text{for all $x,y\in \R$ and a.e. $t>0$}.
\ee
The second condition is {\it  strong initial continuity of the energy}. That is, 
\be\label{SIC}
\lim_{t\to 0^+}\int_\R g(x)v(x,t)^2d\rho_t(x)=\int_\R g(x)v_0(x)^2d\rho_0(x)
\ee
for $g\in C_c(\R)$. It turns out that a weak solution pair with these two properties is unique. This was first proved by Huang and Wang \cite{HuangWang2001}.  See also \cite{LiWarnecke2002, LiWarnecke2003, NguyenTudorascu2015,MR3039208, CarrilloGaltung2023, SuderTudorascu2023,Suder2021Thesis,HuangTang2023EulerPoisson} for subsequent discussions of uniqueness of weak solution pairs for \eqref{SPS} and related systems.    

\par In what follows, we will revisit this uniqueness question. In particular, we will consider this question for solutions which satisfy the  {\it  entropy condition}
\be\label{entropy}
(v(x,t)-v(y,t))(x-y)\le\frac{1}{t}(x-y)^2\; \text{for $\rho_t$ a.e. $x,y\in \R$ and a.e. $t>0$}
\ee
and the {\it initial kinetic energy bound}
\be\label{SICtwo}
\int_{\R}v(x,t)^2d\rho_t(x)\le \int_{\R}v_0(x)^2d\rho_0(x)
\text{ for almost every $t\ge 0$. }
\ee
Note that \eqref{Eentropy} is a stronger condition than \eqref{entropy}. It also turns out that \eqref{SICtwo} implies that \eqref{SIC} holds with the limit understood as an essential limit for $g\in C_b(\R)$; see Lemma \ref{StrongRightCont} below. Regardless, we will work directly with \eqref{entropy} and \eqref{SICtwo} since weak solution pairs which have these properties have been designed \cite{MR4007615,Hynd2020AnnaliSticky,MR2540269}.  
\\
\par {\bf Main result}.  In order to verify our main theorem below, we will make one additional simplifying assumption which is 
\be\label{SecondFiniteMoment}
\int_{\R}x^2d\rho_0(x)<\infty. 
\ee
By two weak solution pairs $\rho, v$ and  $\overline\rho, \overline v$ agreeing, we mean that 
$\rho_t=\overline{\rho}_t$ for all $t\ge 0$ and $v(x,t)=\overline v(x,t)$ for $\rho_t$ almost every $x\in \R$ and almost every $t\ge 0$.  
\\\\
\noindent{\bf Uniqueness Theorem}. {\it Suppose $\rho_0\in {\cal P}(\R)$ satisfies \eqref{SecondFiniteMoment} and $v_0\in L^2(\rho_0)$. There can be at most one weak solution pair $\rho$ and $v$ of the pressureless Euler equations \eqref{SPS} which has initial conditions $\rho_0$ and $v_0$, satisfies the entropy condition \eqref{entropy}, and fulfills the initial kinetic energy bound \eqref{SICtwo}.}
\\
\par Let us consider a Lipschitz continuous $X:[0,\infty)\to L^2(\rho_0)$ such that
\be\label{FlowMapIVP}
\begin{cases}
\displaystyle\dot X(t)=\E_{\rho_0}[v_0|X(t)]\text{ a.e. $t\ge 0$}\\
\displaystyle X(0)=\text{id}_{\R}.
\end{cases}
\ee
Here $\E_{\rho_0}[v_0|X(t)]$ is the conditional expectation of $v_0$ given $X(t)$ which is also the $L^2(\rho_0)$ projection of $v_0$ onto the subspace of $X(t)$ measurable functions on $\R$. The differential equation in \eqref{FlowMapIVP} is related to the pressureless Euler equations as it encodes the averaging which happens during perfectly inelastic collisions \cite{MR4007615}. Dermoune introduced this initial value problem in \cite{Dermoune1999StickyProbab} and explained that if we can solve it for $X$, we can choose a Borel $v: \R\times(0,\infty)\to \R$ such that 
\be\label{dermoune-choice}
\rho_t=X(t)_{\#}\rho_0\quad \text{and} \quad v(X(t),t)=\E_{\rho_0}[v_0|X(t)]
\ee
is a weak solution pair of \eqref{SPS} with the initial conditions $\rho_0$ and $v_0$.

\par Since 
$$
\dot X(t)=v(X(t),t)\text{ a.e. $t\ge 0$}
$$
with the $v$ chosen above, we say that $X$ is a {\it Lagrangian coordinate} corresponding to the pressureless Euler equations.  In what follows, 
we will show that {\it every} weak solution of the pressureless Euler equations which satisfies the entropy condition and the initial kinetic energy bound gives rise to such a Lagrangian coordinate. Then we will establish, in a certain sense, that the initial value problem \eqref{FlowMapIVP} has a unique solution.  The key device in that proof is the monotonicity inequality 
\be\label{KeyMonotonicityIneq}
\int_\R(\E_\rho[Z|X]-\E_\rho[Z|Y])(X-Y)d\rho \ge0,
\ee
which holds for $\rho\in {\cal P}(\R)$ and nondecreasing $X,Y,Z\in L^2(\rho)$.
\\
\par {\bf Our contribution.}  Our aim is to give a
direct approach to the uniqueness question
studied by Huang and Wang in 2001 \cite{HuangWang2001}.  Huang and Wang’s result allows more general initial mass measures $\rho_0$. However, their proof required bounded initial velocities $v_0$. Here we assume that $\rho_0$ has a finite second moment but only suppose $v_0\in L^2(\rho_0)$. 

\par Uniqueness of solutions to the pressureless Euler equations for given initial conditions remains of interest. Recently,  Suder and Tudorascu \cite{SuderTudorascu2023} use the Huang and Wang uniqueness criteria to identify the sticky particle flows considered there with scalar conservation law solutions for bounded, right-continuous initial velocities.  Likewise, Carrillo and Galtung \cite{CarrilloGaltung2023} established equivalence between gradient flow solutions and entropy solutions of associated scalar
conservation laws for pressureless Euler systems, including
systems with interaction forces.  

\par As mentioned above, we derive a Lagrangian coordinate from
every weak solution pair satisfying the entropy condition and initial kinetic energy bound and prove uniqueness using \eqref{KeyMonotonicityIneq}. Our proof is independent of earlier uniqueness results. We anticipate that this method may also
be useful for pressureless Euler systems with interaction potentials.
\\
\par This paper is organized as follows. In the next section, we recall a few key properties of weak solutions and prove the strong initial continuity of energy. We also recall some properties of conditional expectation and prove  inequality \eqref{KeyMonotonicityIneq}. In Section \ref{sect:prob-measure}, we show that corresponding to any weak solution pair of \eqref{SPS} is a Borel probability measure $\eta$ on the space of continuous functions from $[0,\infty)$ into $\R$ with the topology of local uniform convergence. Then we show in Section \ref{sect:lagrangian} that this $\eta$ is determined by an $X$ which solves \eqref{FlowMapIVP}. Finally, we will use this observation to prove the Uniqueness Theorem in Section \ref{sect:uniqueness}.

\section{Preliminaries}\label{sect:prelim}
In this section, we will discuss a few basic properties of weak solutions and conditional expectation that will be used throughout this note.

\subsection{Strong initial continuity of energy}
First, we recall some equivalent interpretations of the conservation of momentum \eqref{EP2weak}. As the following claim is a fairly routine exercise with test functions, we will omit the proof. 
\begin{prop}\label{prop:cons-mom-equiv} Suppose $\rho$ and $v$ is a weak solution pair of the pressureless Euler equations \eqref{SPS} which has initial conditions $\rho_0\in {\cal P}(\R)$ and $v_0\in L^2(\rho_0)$.  The following are equivalent. 
\begin{enumerate}[(i)]

\item For almost every $\tau\ge 0$,
 $$
\int_\R \phi(x,\tau)v(x,\tau)d\rho_\tau(x)=\int_\R \phi(x,0)v_0(x)d\rho_0(x)+\int^\tau_0\int_{\R}v(\partial_t\phi+v\partial_x\phi)d\rho_tdt
 $$
 for all $\phi\in C^1_b(\R\times[0,\infty))$.
 
 \item Equation \eqref{EP2weak} holds for all $\phi\in C^1_c(\R\times[0,\infty))$. 

 \item For almost every $\tau\ge 0$,
 \be
\int_\R g(x)v(x,\tau)d\rho_\tau(x)=\int_\R g(x)v_0(x)d\rho_0(x)+\int^\tau_0\int_{\R}v(x,t)^2g'(x)d\rho_t(x)dt
 \ee
 for each $g\in C^1_b(\R)$.
 \end{enumerate}
 \end{prop}
 
\par Note that for a Lebesgue measurable function $f: (0,\infty)\to \R$, we'll write
$$
\esslim_{t\to 0^+} f(t)=L
$$
provided that for every \(\varepsilon>0\), there exists \(\delta>0\) such that
$|f(t)-L|<\varepsilon$ for almost every $t$ with $0<t<\delta.$ It is equivalent to 
require that there is a null set \(N\subset [0,\infty)\) such that 
$$
\lim_{\substack{t\to 0^+\\ t\not\in N}}f(t)=L. 
$$
\begin{lem}\label{StrongRightCont} 
Suppose $\rho$ and $v$ is a weak solution pair of the pressureless Euler equations \eqref{SPS} which has initial conditions $\rho_0\in {\cal P}(\R)$ and $v_0\in L^2(\rho_0)$ and additionally satisfies the initial kinetic energy bound \eqref{SICtwo}.  
The limit  
\be\label{eq:limit-sic-2}
\esslim_{t \to 0^+}\int_{\R}g(x)v(x,t)^2d\rho_t(x)=\int_{\R}g(x)v_0(x)^2d\rho_0(x)
\ee
holds for each $g\in C_b(\R)$.
\end{lem}
\begin{proof}
Choose $N^c\subset [0,\infty)$ as a subset of full measure such that for all $t\not\in N$, 
$$
\int_{\R}g(x)v(x,t)d\rho_t(x)=\int_{\R}g(x)v_0(x)d\rho_0(x)+\int^t_0\int_{\R}g'(x)v(x,s)^2d\rho_s(x)ds
$$
for each $g\in C^1_b(\R)$. We may also assume 
\be\label{SICthree}
\int_{\R}v(x,t)^2d\rho_t(x)\le \int_{\R}v_0(x)^2d\rho_0(x)
\ee
for all $t\not\in N$ since this holds on a set of full Lebesgue measure.

\par It follows that 
$$
\lim_{\substack{t\to0^+\\t\not\in N}}\int_{\R}g(x)v(x,t)d\rho_t(x)=\int_{\R}g(x)v_0(x)d\rho_0(x).
$$
for each $g\in C^1_b(\R)$.  According to Theorem 5.4.4 of \cite{AmbrosioGigliSavare2008}, 
$$
\liminf_{\substack{t\to0^+\\t\not\in N}}\int_{\R}v(x,t)^2d\rho_t(x)\ge \int_{\R}v_0(x)^2d\rho_0(x).
$$
And by \eqref{SICthree}, 
\be
\limsup_{\substack{t\to0^+\\t\not\in N}}\int_{\R}v(x,t)^2d\rho_t(x)\le \int_{\R}v_0(x)^2d\rho_0(x).
\ee
Therefore, 
\be\label{RealSIC0}
\lim_{\substack{t\to0^+\\t\not\in N}}\int_{\R}v(x,t)^2d\rho_t(x)= \int_{\R}v_0(x)^2d\rho_0(x).
\ee

Now let $t_k\in N^c$ with $t_k \to 0$. Theorem 5.4.4 of \cite{AmbrosioGigliSavare2008} also gives that 
 $$
 (\id \times v(\cdot, t_k))_{\#}\rho_{t_k}\to  (\id \times v_0)_{\#}\rho_0
 $$
 as $k\to \infty$  narrowly in ${\cal P}(\R\times \R)$.  It is now standard to employ \eqref{RealSIC0} to show that 
 $$
 (x,y)\mapsto g(x)y^2
 $$  
 is uniformly integrable with respect to the family $\{ (\id \times v(\cdot, t_k))_{\#}\rho_{t_k} \}_{k\in\N}\subset {\cal P}(\R\times \R)$ provided that $g\in C_b(\R)$. 
 
 \par Consequently, 
 \begin{align}
 \lim_{k\to\infty}\int_{\R}g(x)v(x,t_k)^2d\rho_{t_k}(x)&=
 \lim_{k\to\infty}\int_{\R^2}g(x)y^2d\left[(\id \times v(\cdot, t_k))_{\#}\rho_{t_k}\right](x,y)\\
  &= \int_{\R^2}g(x)y^2d\left[(\id \times v_0)_{\#}\rho_0\right](x,y)\\
  &=\int_{\R}g(x)v_0(x)^2d\rho_0(x).
 \end{align}
As the sequence $t_k$ was arbitrary, the limit \eqref{eq:limit-sic-2} holds for each $g\in C_b(\R)$.
\end{proof}

\subsection{Conditional expectation}\label{sect:conditional}
Suppose $\rho\in {\cal P}(\R)$, $Y\in L^2(\rho)$, and $X:\R\to [-\infty,\infty]$ is a nondecreasing function. Recall that the conditional expectation $\E_\rho[Y|X]$ is any function
$f(X)\in L^2(\rho)$, where $f: [-\infty,\infty]\to\R$ is Borel and
$$
\int_{\R}f(X) h(X)d\rho=\int_{\R}Y h(X)d\rho
$$
for each Borel $h: [-\infty,\infty]\to\R$ with $h(X)\in L^2(\rho)$.

\par It is a standard fact that $\E_\rho[Y|X]$ exists and is uniquely specified $\rho$ almost everywhere. Moreover,
$$
\int_{\R}(\E_\rho[Y|X]-Y)^2d\rho\le \int_{\R}(h(X)-Y)^2d\rho
$$
for each Borel $h: [-\infty,\infty]\to\R$ with $h(X)\in L^2(\rho)$. Therefore, $\E_\rho[Y|X]$ is the projection of $Y$ onto the subspace of $L^2(\rho)$ of $X$ measurable functions. We also note that the notion of conditional expectation applies to general probability spaces as discussed in any standard reference on probability theory.

\par It is routine to check that if $X$ is increasing, then
$$
\E_\rho[Y|X]=Y.
$$
If $X$ is not increasing, there are at most countably many $c\in [-\infty,\infty]$ where
$
X^{-1}(\{c\})
$
has more than one point. The key thing to note here is that these fibers of $X$ are disjoint intervals; since these intervals are nontrivial, they each contain a rational number. As a result, the collection of such intervals must be countable.

\par We will index these intervals with $i\in I\subset \N$ and
write
$$
C_i=X^{-1}(\{c_i\})
$$
for a finite or infinite sequence of distinct values $(c_i)_{i\in I}$.
Note in particular, $C_i$ is a nontrivial interval, possibly open, closed, half-open, or unbounded.
We will also write
$$
O=\R\setminus \bigcup_{i\in I}C_i
$$
and note that $X$ is increasing on $O$.

\par The following assertion gives a characterization of $\E_\rho[Y|X]$. In the proof, we will also consider the generalized inverse of $X$ which is defined by
$$
X^{-1}(r)=\inf\bigg\{y\in\R:X(y)\ge r\bigg\},
$$
for $r\in [-\infty,\infty]$. Note that $X^{-1}$ is nondecreasing and $X^{-1}(X(y))\le y$ for any $y\in \R$. Moreover, if $y\in O$, then $X^{-1}(X(y))=y$.

\begin{lem}\label{LemmaCondExpFormula}
For each $i\in I$,
$$
\E_\rho[Y|X]=\frac{1}{\rho(C_i)}\int_{C_i}Yd\rho
$$
$\rho$ almost everywhere on $C_i$ with $\rho(C_i)>0$ and
$$
\E_\rho[Y|X]=Y
$$
$\rho$ almost everywhere on $O$.
\end{lem}
\begin{proof}
Define $Z: \R\rightarrow \R$ via
$$
Z=
\begin{cases}
\displaystyle\frac{1}{\rho(C_i)}\int_{C_i}Yd\rho& \text{on $C_i$ with $\rho(C_i)>0$}
\\
0& \text{on $C_i$ with $\rho(C_i)=0$}
\\
\displaystyle Y & \text{on $O$}.
\end{cases}
$$
Since
$$
Z=Y\circ X^{-1}\circ X
$$
on $O$ and both $X$ and $Z$ are constant on all $C_i$, it follows that $Z$ is $X$ measurable. That is, $Z=h(X)$ for some Borel $h:[-\infty,\infty]\to\R$.

\par It is straightforward to invoke Jensen's inequality and show
$$
\int_\R Z^2d\rho\le \int_\R Y^2d\rho.
$$
Furthermore,
if $g:[-\infty,\infty]\to\R$ is Borel with $g(X)\in L^2(\rho)$, then
\begin{align}
\int_{\R}g(X)Zd\rho&=\sum_{i\in I^+}\int_{C_i}g(X)Zd\rho+\int_{O}g(X)Zd\rho\\
&=\sum_{i\in I^+}g(c_i)\int_{C_i}Zd\rho+\int_{O}g(X)Yd\rho\\
&=\sum_{i\in I^+}g(c_i)\int_{C_i}Yd\rho+\int_{O}g(X)Yd\rho\\
&=\sum_{i\in I^+}\int_{C_i}g(X)Yd\rho+\int_{O}g(X)Yd\rho\\
&=\int_{\R}g(X)Yd\rho.
\end{align}
Here $I^+$ are the indices $i\in I$ with $\rho(C_i)>0$. The assertion follows as $\E_\rho[Y|X]$ is defined up to a set of $\rho$ measure $0$.
\end{proof}

\subsection{Monotonicity inequality}
Next we will verify inequality \eqref{KeyMonotonicityIneq}.   When we say that $Y\in L^2(\rho)$ is nondecreasing for a given $\rho\in {\cal P}(\R)$, we mean that there is a nondecreasing $\tilde Y: \R\to[-\infty,\infty]$ 
with $Y=\tilde Y$ $\rho$ almost everywhere. We will also say that $Y$ is increasing if $\tilde Y$ is increasing. 

\begin{lem}
Suppose $\rho\in {\cal P}(\R)$ and $X,Y, Z\in L^2(\rho)$ are nondecreasing. Then 
 \be
 \int_\R(\E_\rho[Z|X]-\E_\rho[Z|Y])(X-Y)d\rho \ge0.
 \ee
\end{lem}
\begin{proof}
By the defining property of conditional expectation,
\begin{align*}
&\int_\R
\big(\E_\rho[Z|X]-\E_\rho[Z|Y]\big)(X-Y)\,d\rho\\
&\qquad=
\int_\R \big(Z-\E_\rho[Z|X]\big)Y\,d\rho
+\int_\R \big(Z-\E_\rho[Z|Y]\big)X\,d\rho.
\end{align*}
We will show that both integrals on the right are nonnegative. 

\par Without any loss of generality, we may assume that $X: \R\to [-\infty,\infty]$ is nondecreasing. Consequently, there is a countable family of 
disjoint intervals $(C^x_{i})_{i\in I^x}$ for which $X|_{C^x_i}$ is equal to a constant $c_i^x$ for all $i\in I^x$.  And $X|_{O^x}$ is increasing, where
$$
O^x=\R\setminus \bigcup_{i\in I^x}  C^x_i.
$$
Moreover, 
\be\label{ERhoIdXformula}
\E_\rho[Z|X]=
\begin{cases}
\displaystyle\frac{1}{\rho(C^x_i)}\int_{C^x_i}Z d\rho& \text{on $C^x_i$ with $\rho(C^x_i)>0$}
\\
\displaystyle Z & \text{in $O^x$}
\end{cases}
\ee
$\rho$ almost everywhere by Lemma \ref{LemmaCondExpFormula}.

\par Observe
$$
\int_\R \big(Z-\E_\rho[Z|X]\big)Y\,d\rho
=
\sum_{\substack{i\in I^x\\ \rho(C_i^x)>0}}
\int_{C_i^x}
\left(Z-\frac{1}{\rho(C_i^x)}
\int_{C_i^x}Z\,d\rho\right)Y\,d\rho.
$$
For each $i\in I^x$ with $\rho(C_i^x)>0$, we have
\begin{align*}
&\int_{C_i^x}
\left(Z-\frac{1}{\rho(C_i^x)}
\int_{C_i^x}Z\,d\rho\right)Y\,d\rho\\
&\qquad=
\int_{C_i^x}ZY\,d\rho
-\frac{1}{\rho(C_i^x)}
\left(\int_{C_i^x}Z\,d\rho\right)
\left(\int_{C_i^x}Y\,d\rho\right)\\
&\qquad=
\frac{1}{2\rho(C_i^x)}
\int_{C_i^x}\int_{C_i^x}
\big(Z(a)-Z(w)\big)\big(Y(a)-Y(w)\big)
\,d\rho(a)\,d\rho(w)\\
&\qquad\ge 0,
\end{align*}
since $Y$ and $Z$ are nondecreasing. It follows that
$$
\int_\R \big(Z-\E_\rho[Z|X]\big)Y\,d\rho\ge 0.
$$
Interchanging $X$ and $Y$ gives
$$
\int_\R \big(Z-\E_\rho[Z|Y]\big)X\,d\rho\ge 0.
$$
Combining these inequalities with the first identity proves
the assertion.
\end{proof}

\section{Path space}\label{sect:prob-measure}
We will denote $\Gamma$ as the space of continuous paths $\gamma: [0,\infty)\to \R$ endowed with the topology of local uniform convergence. It is well known that $\Gamma$ is metrizable. Moreover, it is also a standard fact that $\Gamma$ is complete and separable. We will adapt Theorem 8.2.1 of \cite{AmbrosioGigliSavare2008} and show that solutions of \eqref{SPS} can be characterized by a Borel probability measure on $\Gamma$. 

\par To this end, we will make use of the evaluation maps 
$$
e_t:\Gamma\to \R; \gamma\to \gamma(t)
$$
for $t\ge 0$.  We recall each $e_t$ is continuous and therefore measurable.  We will also write 
$$
AC^2_{\text{loc}}([0,\infty))
$$
for the paths $\gamma:[0,\infty)\to \R$ for which $\gamma:[0,T]\to \R$ is absolutely continuous with $\dot\gamma\in L^2([0,T])$ for each $T>0$. 

\begin{lem}\label{EtaLemma}
Suppose $\rho$ and $v$ is a weak solution pair of the pressureless Euler equations \eqref{SPS} which has initial conditions $\rho_0\in {\cal P}(\R)$ and $v_0\in L^2(\rho_0)$ and additionally satisfies \eqref{entropy} and \eqref{SICtwo}.  There is $\eta\in {\cal P}(\Gamma)$ with the following properties. 

\begin{enumerate}[(i)]

\item  $\eta$ is concentrated on $AC^2_{\text{loc}}([0,\infty))$.

\item For $\eta $ almost every $\gamma\in\Gamma$,
$\dot\gamma(t)=v(\gamma(t),t)$ a.e. $t\ge 0$.

\item $\rho_t=e_t{_\#}\eta$ for each $t\ge 0$.

\item For every $\gamma, \zeta\in \textup{supp}(\eta)$
$$
\frac{|\gamma(t)-\zeta(t)|}{t}\le \frac{|\gamma(s)-\zeta(s)|}{s}
$$
for $0<s\le t$. 

\item For each $\gamma, \zeta\in \textup{supp}(\eta)$ with $\gamma(0)<\zeta(0)$, 
$\gamma(t)\le \zeta(t)$ for $t\ge 0$.

\item 
$$
\lim_{t\to0^+}\int_\Gamma\left(\frac{\gamma(t)-\gamma(0)}{t}-v_0(\gamma(0))\right)^2d\eta(\gamma)=0.
$$
\end{enumerate}

\end{lem}
\begin{proof}
1. Parts $(i)$, $(ii)$, and $(iii)$ follow from an adaptation of Theorem 8.2.1 of
\cite{AmbrosioGigliSavare2008}. As a result, we will omit the details. 

\par 2. It is routine to check that the entropy condition \eqref{entropy} implies $(\rho_t\otimes\rho_t)(E_t)=1$ for almost every $t>0$, where 
$$
E_t=\left\{ (x,y)\in \R^2: (v(x,t)-v(y,t))(x-y)\le\frac{1}{t}(x-y)^2\right\}.
$$
By part $(iii)$, 
\begin{align}
1&=\int^\infty_0\int_{\R^2}\chi_{E_t}(x,y)e^{-t}d(\rho_t\otimes\rho_t)(x,y)dt=\int^\infty_0\int_{\Gamma^2}\chi_{E_t}(\gamma(t),\zeta(t))e^{-t} d(\eta\otimes\eta)(\gamma,\zeta)dt.
\end{align} 
It follows that $(\gamma(t),\zeta(t))\in E_t$ for $\eta\otimes \eta$ almost every $(\gamma,\zeta)\in\Gamma^2$ and almost every $t>0$. And by parts $(i)$ and $(ii)$, 
\begin{align}
\frac{d}{dt}\frac{(\gamma(t)-\zeta(t))^2}{t^2}
&=\frac{2}{t^2}
\bigg[
\big(v(\gamma(t),t)-v(\zeta(t),t)\big)
\big(\gamma(t)-\zeta(t)\big)
-\frac{(\gamma(t)-\zeta(t))^2}{t}
\bigg]\\
&\le 0
\end{align}
for $\eta\otimes \eta$ almost every $(\gamma,\zeta)\in\Gamma^2$ and almost every $t>0$.    

\par Consequently,
\be\label{eq:quant-stickiness}
\frac{|\gamma(t)-\zeta(t)|}{t}
\le
\frac{|\gamma(s)-\zeta(s)|}{s}
\ee
for $\eta\otimes\eta$ almost every $(\gamma,\zeta)\in \Gamma^2$ and each fixed $0<s\le t$. Note that both sides of the above inequality are continuous in $(\gamma,\zeta)$. Therefore, the inequality holds for all $(\gamma,\zeta)\in \textup{supp}(\eta\otimes\eta)$. 
 And as
$$
\textup{supp}(\eta\otimes\eta)=\textup{supp}(\eta)\times\textup{supp}(\eta),
$$
the inequality actually holds for all
$\gamma,\zeta\in\textup{supp}(\eta)$. This proves $(iv)$.

\par Suppose now that $\gamma,\zeta\in\text{supp}(\eta)$
and $\gamma(0)<\zeta(0)$. If $\gamma(t)>\zeta(t)$ for
some $t>0$, continuity gives $s\in(0,t)$ with
$\gamma(s)=\zeta(s)$.  Inequality \eqref{eq:quant-stickiness} then
implies $\gamma(t)=\zeta(t)$, which is a contradiction.
Thus $\gamma(t)\le\zeta(t)$ for all $t\ge 0$, and we conclude $(v)$.

\par 3. We are left to verify $(vi)$.  We first show that
$$
\lim_{\substack{t\to0^+\\t\not\in N}}
\int_\Gamma
\big(v(\gamma(t),t)-v_0(\gamma(0))\big)^2d\eta(\gamma)
=0
$$
for a suitable Lebesgue null set $N\subset[0,\infty)$.
Choose $N$ so that the identity in part $(iii)$ of
Proposition \ref{prop:cons-mom-equiv} and the initial
kinetic energy bound hold for every $t\not\in N$.
By $(iii)$,
$$
\int_\Gamma v(\gamma(t),t)^2d\eta(\gamma)
=\int_\R v(x,t)^2d\rho_t(x)
\le\int_\R v_0^2d\rho_0
$$
for these times.

\par Fix $h\in C^1_c(\R)$. Part $(iii)$ of Proposition
\ref{prop:cons-mom-equiv} implies
$$
\lim_{\substack{t\to0^+\\t\not\in N}}
\int_\Gamma h(\gamma(t))v(\gamma(t),t)d\eta(\gamma)
=\int_\R hv_0d\rho_0.
$$
Dominated convergence also gives
$$
\lim_{t\to0^+}
\int_\Gamma
|h(\gamma(t))-h(\gamma(0))|^2d\eta(\gamma)=0.
$$
The Cauchy--Schwarz inequality then implies
\begin{align}
&\left|
\int_\Gamma
\big(h(\gamma(t))-h(\gamma(0))\big)
v(\gamma(t),t)d\eta(\gamma)
\right|\\
&\qquad\le
\left(
\int_\Gamma
|h(\gamma(t))-h(\gamma(0))|^2d\eta(\gamma)
\right)^{1/2}
\left(\int_\R v_0^2d\rho_0\right)^{1/2}
\to0
\end{align}
as $t\to0^+$ with $t\not\in N$. It follows that
$$
\lim_{\substack{t\to0^+\\t\not\in N}}
\int_\Gamma h(\gamma(0))v(\gamma(t),t)d\eta(\gamma)
=\int_\R hv_0d\rho_0.
$$

\par 4. The Cauchy--Schwarz inequality and the initial
kinetic energy bound give
\begin{align}
&\left|
\int_\Gamma
\big(v_0(\gamma(0))-h(\gamma(0))\big)
v(\gamma(t),t)d\eta(\gamma)
\right|\\
&\qquad\le
\left(
\int_\Gamma
|v_0(\gamma(0))-h(\gamma(0))|^2d\eta(\gamma)
\right)^{1/2}
\left(
\int_\Gamma v(\gamma(t),t)^2d\eta(\gamma)
\right)^{1/2}\\
&\qquad\le
\|v_0-h\|_{L^2(\rho_0)}\|v_0\|_{L^2(\rho_0)}
\end{align}
for $t\not\in N$. In particular, this estimate is
uniform in $t\not\in N$. Also observe
\begin{align}
&\left|
\int_\Gamma v_0(\gamma(0))v(\gamma(t),t)d\eta(\gamma)
-\int_\R v_0^2d\rho_0
\right|\\
&\qquad\le
\left|
\int_\Gamma
\big(v_0(\gamma(0))-h(\gamma(0))\big)
v(\gamma(t),t)d\eta(\gamma)
\right|\\
&\qquad\quad+
\left|
\int_\Gamma h(\gamma(0))v(\gamma(t),t)d\eta(\gamma)
-\int_\R hv_0d\rho_0
\right|+
\left|\int_\R(h-v_0)v_0d\rho_0\right|\\
&\qquad\le
2\|v_0-h\|_{L^2(\rho_0)}\|v_0\|_{L^2(\rho_0)}+
\left|
\int_\Gamma h(\gamma(0))v(\gamma(t),t)d\eta(\gamma)
-\int_\R hv_0d\rho_0
\right|.
\end{align}
Here we applied the preceding estimate to the first
term and the Cauchy--Schwarz inequality to the third.

\par For each fixed $h\in C^1_c(\R)$, the last
absolute value tends to zero as $t\to0^+$ with
$t\not\in N$. Consequently,
\begin{align}
\limsup_{\substack{t\to0^+\\t\not\in N}}
\left|
\int_\Gamma v_0(\gamma(0))v(\gamma(t),t)d\eta(\gamma)
-\int_\R v_0^2d\rho_0
\right|\le 2\|v_0-h\|_{L^2(\rho_0)}\|v_0\|_{L^2(\rho_0)}.
\end{align}
Since $C^1_c(\R)$ is dense in $L^2(\rho_0)$, the
right hand side can be made arbitrarily small.
We conclude that
$$
\lim_{\substack{t\to0^+\\t\not\in N}}
\int_\Gamma v_0(\gamma(0))v(\gamma(t),t)d\eta(\gamma)
=\int_\R v_0^2d\rho_0.
$$
It follows that
\begin{align}
\int_\Gamma
\big(v(\gamma(t),t)-v_0(\gamma(0))\big)^2d\eta(\gamma)
&\le
2\int_\R v_0^2d\rho_0
-2\int_\Gamma
v_0(\gamma(0))v(\gamma(t),t)d\eta(\gamma)
\to0
\end{align}
as $t\to0^+$ with $t\not\in N$.

\par 5. Finally, $(i)$ and $(ii)$, together with Jensen's
inequality and Tonelli's theorem, give
\begin{align}
\int_\Gamma
\left(
\frac{\gamma(t)-\gamma(0)}{t}-v_0(\gamma(0))
\right)^2d\eta(\gamma)\le
\frac{1}{t}\int^t_0\int_\Gamma
\big(v(\gamma(s),s)-v_0(\gamma(0))\big)^2
d\eta(\gamma)ds
\end{align}
for every $t>0$. The inner integral has essential
limit zero as $s\to0^+$ and is bounded above by
$4\int_\R v_0^2d\rho_0$ for almost every $s>0$.
Therefore,  we can apply dominated convergence and deduce $(vi)$.
\end{proof}
\begin{cor}
There is a sequence $(t_k)_{k\in \N}$ of positive numbers which decreases to $0$ such that 
\be\label{PointwiseDifferentiabilityZero}
\lim_{k\to \infty}\frac{\gamma(t_k)-\gamma(0)}{t_k}=v_0(\gamma(0))
\ee
for $\eta$ almost every $\gamma\in \Gamma$. 
\end{cor}
\begin{proof}
This assertion follows from part $(vi)$ of the above lemma as $L^2(\eta)$ convergence implies almost everywhere convergence up to a subsequence. 
\end{proof}

\section{Lagrangian coordinate}\label{sect:lagrangian}
A basic assertion involving monotone functions is as follows. 
\begin{lem}
Suppose $A\subset \R$ and $f: A\to \R$ is nondecreasing. There is a nondecreasing $F: \R\to [-\infty,\infty]$ with $F|_{A}=f$. 
\end{lem}
\begin{proof}
For $x\in \R$, define
\[
F(x)=\sup\{f(a):a\in A,\ a\le x\},
\]
with the convention that \(\sup\varnothing=-\infty\). Then \(F:\mathbb R\to[-\infty,\infty]\). Also notice that if \(x<y\), then
\[
\{a\in A:a\le x\}\subseteq \{a\in A:a\le y\},
\]
so \(F(x)\le F(y)\). Thus \(F\) is nondecreasing.

Now let \(x\in A\). Since \(x\in A\) and \(x\le x\), we have \(F(x)\ge f(x)\). On the other hand, if \(a\in A\) and \(a\le x\), then the monotonicity of \(f\) gives \(f(a)\le f(x)\). Therefore
\[
F(x)=\sup\{f(a):a\in A,\ a\le x\}\le f(x).
\]
Hence \(F(x)=f(x)\). Therefore \(F|_A=f\).
\end{proof}

\par Suppose $\rho_0\in {\cal P}(\R)$ satisfies 
$$
\int_{\R}x^2d\rho_0(x)<\infty 
$$ 
and $v_0\in L^2(\rho_0)$.  For the remainder of this section, we will fix a weak solution pair $\rho$ and $v$ of \eqref{SPS} with initial conditions $\rho_0$ and $v_0$ and which satisfies \eqref{entropy} and \eqref{SICtwo}.
We will use the above lemma and Lemma \ref{EtaLemma} to show that there is a mapping $X$ associated with $\rho$ and $v$. 
\begin{prop}\label{prop:goodtrajectories}
There is a Lipschitz continuous 
$$
X: [0,\infty)\to L^2(\rho_0)
$$
with $X(t)$ nondecreasing for each $t\ge 0$, $\rho_t=X(t)_{\#}\rho_0$ for $t\ge 0$, and
\be
\begin{cases}
\dot X(t)=v(X(t),t), \text{a.e. $t\ge 0$}\\
X(0)=\id.
\end{cases}
\ee
\end{prop}
\begin{proof}
According to Lemma \ref{EtaLemma}, there is a Borel $S\subset \Gamma$ for which $\eta(S)=1$ and $\gamma\in AC^2_{\text{loc}}([0,\infty))$, $\dot \gamma(t)=v(\gamma(t),t)$ for almost every $t\ge 0$, $\gamma(0)< \zeta(0)$ implies $\gamma(t)\le \zeta(t)$ for $t\ge 0$,  and \eqref{PointwiseDifferentiabilityZero} holds for each $\gamma\in S$.  For $t\ge 0$, we define the function 
$$
X: e_0(S)\times[0,\infty)\to \R
$$
via 
\be\label{FlowSinglePath}
X(\gamma(0),t)=\gamma(t)
\ee
for $\gamma\in S$.  We claim that this function is well-defined. 

\par Suppose $\gamma,\zeta\in S$ with
$$
\gamma(0)=\zeta(0)=y.
$$
By our proof of Lemma \ref{EtaLemma} $(iv)$, 
$$
\frac{|\gamma(t)-\zeta(t)|}{t}\le \frac{|\gamma(s)-\zeta(s)|}{s}=
 \left|\frac{\gamma(s)-y}{s}-\frac{\zeta(s)-y}{s}\right|
$$
for $0<s\le t$. Choosing $s=t_k$ and sending $k\to\infty$ gives $\gamma(t)=\zeta(t)$. It follows that $X(t)$ is well-defined. 

\par Furthermore, $X(t): e_0(S)\to \R$ is nondecreasing. By the previous lemma, we extend this function to a monotone $X(t): \R\to [-\infty,\infty]$ that satisfies \eqref{FlowSinglePath} for $\gamma\in S$. 
Observe, if $f\in C_b(\R)$
\begin{align}
\int_\R fd\rho_t&=\int_\Gamma f(\gamma(t))d\eta(\gamma)\\
&=\int_\Gamma f(X(\gamma(0),t))d\eta(\gamma)\\
&=\int_\R f(X(t))d\rho_0.
\end{align}
We conclude that $\rho_t=X(t)_{\#}\rho_0$ for $t\ge 0$.  Also note
\begin{align}
\int_{\Gamma}( \gamma(t)-\gamma(s))^2d\eta(\gamma)
&=\int_{\Gamma}\left( \int^t_s v(\gamma(\tau),\tau)d\tau \right)^2d\eta(\gamma)\\
&\le (t-s)\int^t_s \left(\int_{\R} v(X(\tau),\tau)^2d\rho_0\right) d\tau\\
&= (t-s)\int^t_s \left(\int_{\R} v(x,\tau)^2d\rho_\tau(x)\right) d\tau\\
&\le (t-s)^2\int_{\R} v_0^2d\rho_0.
\end{align}

\par Thus, 
$$
\int_{\R}( X(t)-X(s))^2d\rho_0\le (t-s)^2\int_{\R} v_0^2d\rho_0
$$
for $t,s\ge 0$.  And as 
$$
\int_{\R} X(0)^2d\rho_0=\int_{\R}x^2d\rho_0(x)<\infty,
$$
$X(t)\in L^2(\rho_0)$ for each $t\ge 0$.  Furthermore, $X:[0,\infty)\to L^2(\rho_0)$ is Lipschitz continuous. 

\par In addition, 
\begin{align}
0&=\int_\Gamma \left|\gamma(t)-\gamma(0)-\int^t_0 v(\gamma(s),s)ds\right|d\eta(\gamma)\\
&=\int_\Gamma \left|X(\gamma(0),t)-\gamma(0)-\int^t_0 v(X(\gamma(0),s),s)ds\right|d\eta(\gamma)\\
&=\int_\R \left|X(y,t)-y-\int^t_0 v(X(y,s),s)ds\right|d\rho_0(y).
\end{align}
It follows that $X(y,0)=y$ and $\partial_t X(y,t)=v(X(y,t),t)$ a.e. $t\ge 0$ for $\rho_0$ almost every $y\in \R$. 
\end{proof}

\par We can also show that if $s\le t$ and $X(y,s)=X(z,s)$, then $X(y,t)=X(z,t)$ for $\rho_0$ almost every $y,z$.  This property is implied 
by the following assertion. 
\begin{prop}\label{PropStickyFunction}
Fix $t>0$. There is a Borel measurable function $f: \R\times(0,t]\to \R$ with 
$$
X(t)=f(X(s),s)\;\text{ $\rho_0$ a.e.}
$$
for each $s\in (0,t]$. Moreover, 
\be\label{eq:f-lip-x}
0\le f(x_1,s)-f(x_2,s)\le \frac{t}{s}(x_1-x_2)
\ee
for $x_1\ge x_2$ and $0<s\le t$. 
\end{prop}
\begin{proof}
By Proposition \ref{prop:goodtrajectories}, there is Borel $A\subset \R$ such that $\rho_0(A)=1$, $X(y,\cdot)$ is absolutely continuous, $X(y,0)=y$, and 
$$
\partial_t X(y,t)=v(X(y,t),t), \;\text{a.e. $t\ge 0$}
$$
for all $y\in A$.  Employing part $(iv)$ of Lemma \ref{EtaLemma}, we find
\be\label{QuantitativeStickiness}
\frac{|X(y,t)-X(z,t)|}{t}\le \frac{|X(y,s)-X(z,s)|}{s}
\ee
for $y,z\in A$ and $0< s\le t$.  

\par Set 
$$
f(x,s)=\inf\bigg\{X(y,t)+\frac{t}{s}(x-X(y,s))^+: y\in A\bigg\}
$$
for $(x,s)\in \R\times(0,t]$. We will first show that $f$ is well-defined. Note that for $z\in A$ and $0<s\le t$, 
\be\label{QuantitativeStickinessTwo}
X(z,t)-X(y,t)
\le \frac ts\bigl(X(z,s)-X(y,s)\bigr)^+
\qquad(y\in A).
\ee
Indeed, this follows from \eqref{QuantitativeStickiness} for $y\le z$; and if $y\ge z$, the left-hand side is nonpositive. Moreover, 
$$
\bigl(X(z,s)-X(y,s)\bigr)^+
\le \bigl(X(z,s)-x\bigr)^+
   +\bigl(x-X(y,s)\bigr)^+.
$$
Consequently, 
$$
X(y,t)+\frac ts\bigl(x-X(y,s)\bigr)^+
\ge X(z,t)-\frac ts\bigl(X(z,s)-x\bigr)^+.
$$
It follows that 
$$
f(x,s)\ge X(z,t)-\frac ts\bigl(X(z,s)-x\bigr)^+>-\infty
$$
for $x\in \R$ and $0<s\le t$.  

\par It is also evident that $f$ is upper semicontinuous and that \eqref{eq:f-lip-x} holds as stated.  Furthermore, 
for $z\in A$ and $s\in (0,t]$,
$$
f(X(z,s),s)\le X(z,t)
$$
by choosing $y=z$ in the definition of $f$.  By \eqref{QuantitativeStickinessTwo}, we also have 
$$
X(y,t)+\frac{t}{s}(X(z,s)-X(y,s))^+\ge X(z,t)
$$
for $y\in A$.  Thus, 
$$
f(X(z,s),s)\ge X(z,t).
$$
\end{proof}

\par We can now express the conservation of momentum in terms of $X$. 
\begin{prop}
For almost every $t\ge 0$, $$\dot X(t)=\E_{\rho_0}[v_0|X(t)].$$
\end{prop}
\begin{proof}
Let us choose ${\cal T}\subset (0,\infty)$ of full Lebesgue measure such that the following properties hold for each $t\in {\cal T}$. 

\begin{enumerate}[$(i)$]
\item $v(X(t),t)\in L^2(\rho_0)$ with 
$$
\int_{\R}v(X(t),t)^2d\rho_0\le \int_{\R}v_0^2d\rho_0.
$$

\item $X:[0,\infty)\to L^2(\rho_0)$ is differentiable at $t$ with $\dot X(t)=v(X(t),t)$.

\item For each $g\in C^1_b(\R)$, 
$$
\int_\R g(x)v(x,t)d\rho_t(x)=\int_\R g(x)v_0(x)d\rho_0(x)
+\int^t_0\int_\R g'(x)v(x,s)^2d\rho_s(x)ds.
$$

\item 
$$
\lim_{\delta\to 0}\frac{1}{\delta}\int^{t+\delta}_t\|v(X(\tau),\tau)-v(X(t),t)\|_{L^2(\rho_0)}d\tau=0.
$$
\end{enumerate}
We note that $(iv)$ holds for almost every $t\ge 0$ by Lebesgue differentiation; see Theorem 9 of section 2 in Chapter II of 
\cite{DiestelUhl1977}. 

\par Next fix $h\in C^1_b(\R)$ and $T\in {\cal T}$. We will show 
\be\label{WTSidentity}
\int_{\R}h(X(T))v(X(T),T)d\rho_0=\int_{\R}h(X(T))v_0d\rho_0.
\ee
If we can verify this assertion, we would have 
$$
v(X(T),T)=\E_{\rho_0}[v_0|X(T)]
$$
by the density of $C^1_b(\R)$ in $L^2(\rho_T)$. 

\par In view of Proposition \ref{PropStickyFunction}, we may choose a Borel $f: \R\times (0,T]\to \R$ with 
$$
X(T)=f(X(s),s)\;\text{$\rho_0$ a.e.}
$$
for every $s\in (0,T]$ and
$$
|f(x,s)-f(y,s)|\le \frac{T}{s}|x-y|
$$ 
for $x,y\in \R$ and $s\in (0,T].$ It follows that 
$$
|h(f(x,s))-h(f(y,s))|\le \|h'\|_\infty \frac{T}{s}|x-y|, \quad x,y\in \R. 
$$
Let us fix $t\in {\cal T}\cap(0,T]$ and choose a sequence $g^n\in C^1_b(\R)$ such that 
$$
g^n\to h(f(\cdot, t))\text{ uniformly and } \|(g^n)'\|_\infty\le \|h'\|_\infty \frac{T}{t}.
$$
This sequence can be obtained by using a standard mollifier.

\par Suppose $s\in {\cal T}$ with $s<t$  and note 
$$
\int_\R g^n(x)v(x,t)d\rho_t(x)=\int_\R g^n(x)v(x,s)d\rho_s(x)
+\int^t_s\int_\R (g^n)'(x)v(x,\tau)^2d\rho_\tau(x)d\tau.
$$
In particular, 
\be\label{SpecificConsMom}
\int_\R g^n(X(t))v(X(t),t)d\rho_0=\int_\R g^n(X(s))v(X(s),s)d\rho_0(x)
+\int^t_s\int_\R (g^n)'(X(\tau))v(X(\tau),\tau)^2d\rho_0d\tau.
\ee
We also have 
$$
g^n(X(t))=g^n(X(s))+\int^t_s (g^n)'(X(\tau))v(X(\tau),\tau)d\tau.
$$
Multiplying by $v(X(s),s)$ and integrating gives 
\begin{align}\label{SpecificConsMom2}
&\int_\R g^n(X(t))v(X(s),s)d\rho_0\\
&=\int_\R g^n(X(s))v(X(s),s)d\rho_0(x)
+\int^t_s\int_\R (g^n)'(X(\tau))v(X(\tau),\tau)v(X(s),s) d\rho_0d\tau.
\end{align}

\par Subtracting \eqref{SpecificConsMom2} from \eqref{SpecificConsMom} yields
\begin{align}
&\int_\R g^n(X(t))\left(v(X(t),t)-v(X(s),s)\right)d\rho_0\\
&=\int^t_s\int_\R (g^n)'(X(\tau))v(X(\tau),\tau)(v(X(\tau),\tau)-v(X(s),s))d\rho_0d\tau.
\end{align}
Moreover, 
\begin{align}
&\left|\int_\R g^n(X(t))\left(v(X(t),t)-v(X(s),s)\right)d\rho_0\right|\\
& \le \|h'\|_\infty \frac{T}{t} \int^t_s\|v(X(\tau),\tau)\|_{L^2(\rho_0)}\|v(X(\tau),\tau)-v(X(s),s)\|_{L^2(\rho_0)}d\tau\\
& \le \|h'\|_\infty \frac{T}{t}\|v_0\|_{L^2(\rho_0)} \int^t_s\|v(X(\tau),\tau)-v(X(s),s)\|_{L^2(\rho_0)}d\tau.
\end{align}
And sending $n\to \infty$ gives 
\begin{align}
&\left|\int_\R h(f(X(t),t))\left(v(X(t),t)-v(X(s),s)\right)d\rho_0\right|\\
& \le \|h'\|_\infty \frac{T}{t}\|v_0\|_{L^2(\rho_0)} \int^t_s\|v(X(\tau),\tau)-v(X(s),s)\|_{L^2(\rho_0)}d\tau.
\end{align}
That is, 
\begin{align}\label{FirstMfunEst}
&\left|\int_\R h(X(T))v(X(t),t)d\rho_0 -\int_\R h(X(T))v(X(s),s)d\rho_0 \right|\\
& \le \|h'\|_\infty \frac{T}{t}\|v_0\|_{L^2(\rho_0)} \int^t_s\|v(X(\tau),\tau)-v(X(s),s)\|_{L^2(\rho_0)}d\tau
\end{align}
for $t,s\in {\cal T}$ with $0<s\le t\le T$.

\par Now fix $\delta\in (0,T]$ and set 
$$
m(t)=\int_\R h(X(T))v(X(t),t)d\rho_0
$$
for $t\in{\cal T} \cap [\delta,T]$. Note that 
$$
\|v(X(\tau),\tau)-v(X(s),s)\|_{L^2(\rho_0)}\le \|v(X(\tau),\tau)\|_{L^2(\rho_0)}+\|v(X(s),s)\|_{L^2(\rho_0)}
\le 2\|v_0\|_{L^2(\rho_0)}
$$
for almost every $\tau,s\ge 0$. In view of \eqref{FirstMfunEst}, 
\begin{align}
&\left|m(t)-m(s) \right| \le \|h'\|_\infty \frac{2T}{\delta}\|v_0\|^2_{L^2(\rho_0)} |t-s|
\end{align}
for $t,s\in {\cal T}\cap [\delta,T]$.  Thus, $m: {\cal T}\cap [\delta,T]\to \R$ is a Lipschitz function. 

\par It follows that $m$ has a Lipschitz extension to $[\delta,T]$ which we also denote by $m$. We conclude 
that $m$ is absolutely continuous. Furthermore, \eqref{FirstMfunEst} implies that if $t,s\in {\cal T}\cap (\delta,T)$ with $t> s$,
then
\begin{align}
\left|\frac{m(t)-m(s)}{t-s}\right| \le \|h'\|_\infty \frac{T}{\delta}\|v_0\|_{L^2(\rho_0)} \frac{1}{t-s} \int^t_s\|v(X(\tau),\tau)-v(X(s),s)\|_{L^2(\rho_0)}d\tau.
\end{align}
Since $s\in {\cal T}$, the right hand side tends to $0$ as $t\to s^+$. 

\par Therefore, $m'(s)=0$ for almost every $s\in [\delta,T]$. Consequently, 
$$
\int_\R h(X(T))v(X(T),T)d\rho_0=\int_\R h(X(T))v(X(s),s)d\rho_0
$$
for almost every $s\in [\delta,T]$. As $\delta$ is arbitrary, we can send $s$ to $0$ along an appropriate sequence to conclude 
\eqref{WTSidentity}. Here we are using the strong convergence $v(\gamma(s),s)\to v_0(\gamma(0))$ in $L^2(\eta)$ along a subsequence of 
$s$ tending to $0$ which was verified in our proof of Lemma \ref{EtaLemma}. 
\end{proof}

\section{Uniqueness}\label{sect:uniqueness}
Assume $\rho_0\in {\cal P}(\R)$ has a finite second moment and $v_0\in L^2(\rho_0)$.  Let us also suppose $\rho$ and $v$ is a weak solution pair of \eqref{SPS} with initial conditions \eqref{IC} and which satisfies \eqref{entropy} and \eqref{SICtwo}. As discussed in the previous section, there is a Lipschitz continuous $X: [0,\infty)\to L^2(\rho_0)$   which solves 
\be
\begin{cases}
\displaystyle\dot X(t)=\E_{\rho_0}[v_0|X(t)]\text{ a.e. $t\ge 0$}\\
\displaystyle X(0)=\text{id}_{\R}.
\end{cases}
\ee

\par We will derive a few more properties of $X$ that will allow us to conclude that this weak solution pair is uniquely specified. First we note that $X$ is differentiable at $t=0$. 
\begin{lem}
The limit  
\be\label{differentiabilityZero}
\lim_{t\rightarrow 0^+}\frac{X(t)-\id}{t}=v_0
\ee
holds in $L^2(\rho_0)$. 
\end{lem}
\begin{proof}
This follows from Lemma \ref{EtaLemma} part $(vi)$ as 
$$
\lim_{t\to0^+}\int_{\R}\left(\frac{X(t)- \id}{t}- v_0\right)^2d\rho_0=\lim_{t\to0^+}\int_\Gamma\left(\frac{\gamma(t)-\gamma(0)}{t}-v_0(\gamma(0))\right)^2d\eta(\gamma)=0.
$$
\end{proof}
We also will need the following implicit relation. We remark that this was essentially proved in Lemma 5.1 of \cite{MR4752570}.
\begin{lem}
For each $t\ge 0$,
\be\label{ImplicitRelation}
X(t)=\E_{\rho_0}[\textup{id}_\R+tv_0|X(t)].
\ee
\end{lem}
\begin{proof}
Fix $t>0$ and assume $h: \R\rightarrow \R$ is Borel with $h(X(t))\in L^2(\rho_0)$.  Also recall that 
$X(t)=f(X(s),s)$ for $s\in (0,t]$ (as in Proposition \ref{PropStickyFunction}) and set 
$$
g(s)=\int_{\R}h(X(t))(X(s)-(\text{id}_\R+sv_0))d\rho_0. 
$$
It is not hard to see that $g$ is absolutely continuous with $g(0)=0$.

\par Moreover,  
$$
\dot g(s)=\int_{\R}h(X(t))(\dot X(s)-v_0)d\rho_0=\int_{\R}h(f(X(s),s))(\dot X(s)-v_0)d\rho_0=0
$$
for almost every $s\in (0,t]$. Thus, $g(s)=0$ for all $s\in [0,t]$.  That is, 
$$
\int_{\R}h(X(t))(X(s)-(\text{id}_\R+sv_0))d\rho_0=0
$$
for all $s\in [0, t]$. We conclude by taking $s=t$. 
\end{proof} 
\par We are finally ready to prove the Uniqueness Theorem stated in the introduction.  
\begin{proof}[Proof of the Uniqueness Theorem]
Let $X$ be as above and suppose there is another weak solution pair $\overline \rho$ and $\overline v$ of \eqref{SPS} with initial conditions \eqref{IC} and which satisfies \eqref{entropy} and \eqref{SICtwo}. There is a Lipschitz 
$Y: [0,\infty)\to L^2(\rho_0)$ such that $Y(t)$ is nondecreasing for each $t\ge 0$, 
\be
\begin{cases}
\displaystyle\dot Y(t)=\E_{\rho_0}[v_0|Y(t)]\text{ a.e. $t\ge 0$}\\
\displaystyle Y(0)=\text{id}_{\R},
\end{cases}
\ee
$\overline{\rho}_t=Y(t)_{\#}\rho_0$ for $t\ge 0$, and $\dot Y(t)=\overline v(Y(t),t)$ for almost every $t\ge 0$. 

\par In view of the relation \eqref{ImplicitRelation} and the inequality \eqref{KeyMonotonicityIneq}
\begin{align}
&\frac{d}{dt}\frac{1}{2}\int_\R(X(t)-Y(t))^2d\rho_0\\
&\quad=\int_\R(X(t)-Y(t))(\dot X(t)-\dot Y(t))d\rho_0\\
&\quad=\int_\R(X(t)-Y(t))(\E_{\rho_0}[v_0|X(t)]-\E_{\rho_0}[v_0|Y(t)])d\rho_0\\
&\quad=\frac{1}{t}\int_\R(X(t)-Y(t))(\E_{\rho_0}[tv_0|X(t)]-\E_{\rho_0}[tv_0|Y(t)])d\rho_0\\
&\quad=\frac{1}{t}\int_\R(X(t)-Y(t))(X(t)-\E_{\rho_0}[\id|X(t)]-(Y(t)-\E_{\rho_0}[\id|Y(t)]))d\rho_0\\
&\quad=
\frac{1}{t}\int_\R(X(t)-Y(t))^2d\rho_0- \frac{1}{t}\int_\R(X(t)-Y(t))(\E_{\rho_0}[\text{id}_\R|X(t)]-\E_{\rho_0}[\text{id}_\R|Y(t)])d\rho_0\\
&\quad\le \frac{1}{t}\int_\R(X(t)-Y(t))^2d\rho_0
\end{align} 
for almost every $t>0$.  It follows that 
$$
\frac{d}{dt}\int_\R\left(\frac{X(t)-Y(t)}{t}\right)^2d\rho_0\le 0
$$
for almost every $t>0$. 

\par Consequently,
$$
\int_\R\left(\frac{X(t)-Y(t)}{t}\right)^2d\rho_0\le \int_\R\left(\frac{X(s)-Y(s)}{s}\right)^2d\rho_0
$$
 for $0<s\le t$. Appealing to the limit \eqref{differentiabilityZero}, we also have 
 \begin{align}
\lim_{s\rightarrow 0^+}\int_\R\left(\frac{X(s)-Y(s)}{s}\right)^2d\rho_0&=
\lim_{s\rightarrow 0^+}\int_\R\left(\frac{X(s)-\id}{s}-\frac{Y(s)-\id}{s}\right)^2d\rho_0=0.
 \end{align}
 Thus, 
 $$
 \int_\R\left(\frac{X(t)-Y(t)}{t}\right)^2d\rho_0\le 0
 $$
 and $X(t)=Y(t)$ $\rho_0$ almost everywhere for all $t>0$. 
 
 \par It follows that 
 $$
 \rho_t=X(t)_{\#}\rho_0=Y(t)_{\#}\rho_0= \overline{\rho}_t
 $$
 for each $t\ge 0$. Furthermore, for almost every $t\ge 0$,
 $$
 v(X(t),t)=\dot X(t)=\dot Y(t)= \overline{v}(Y(t),t)=\overline{v}(X(t),t)
 $$
 $\rho_0$ almost everywhere. That is, 
$v(x,t)=\overline{v}(x,t)$ for $\rho_t$ almost every $x\in \R$ and almost every $t\ge 0$.  
\end{proof}


\appendix

\bibliography{SPSUniquebib}{}

\bibliographystyle{plain}

\typeout{get arXiv to do 4 passes: Label(s) may have changed. Rerun}

\end{document}

%% file: noncolliding-figure.tex
\begin{figure}[ht]
\centering
\begin{tikzpicture}[x=1.3cm,y=1.1cm,>=Latex]

\def\Tmax{4}

\draw[->] (-3.2,0) -- (3.4,0) node[right] {$x$};
\draw[->] (0,0) -- (0,\Tmax+0.4) node[above] {$t$};

\foreach \x/\v/\lab in {
    -2.4/-0.45/1,
    -1.6/-0.25/2,
    -0.8/-0.10/3,
     0.8/ 0.10/4,
     1.6/ 0.25/5,
     2.4/ 0.45/6
}{
    \draw[line width=0.4pt] (\x,-0.06) -- (\x,0.06);
    \node[below=3pt] at (\x,0) {$x_{\lab}$};

    \draw[line width=1pt]
        (\x,0) -- ({\x+\v*\Tmax},\Tmax);
}

\end{tikzpicture}
\caption{Particle trajectories for a nondecreasing initial velocity \(v_0\) with point masses located at the positions $x_1,\dots, x_6$.
The particle initially at \(x_i\) follows the straight line
\(x=x_i+t\,v_0(x_i)\), and the trajectories do not intersect. This corresponds to a weak solution of the pressureless Euler equations of the form \eqref{method-of-char}.}
\label{fig:noncolliding-trajectories}
\end{figure}
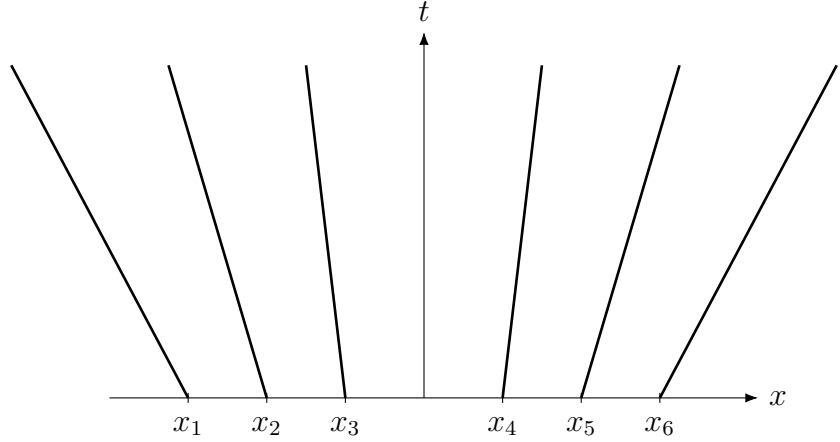

%% file: colliding-schematic.tex
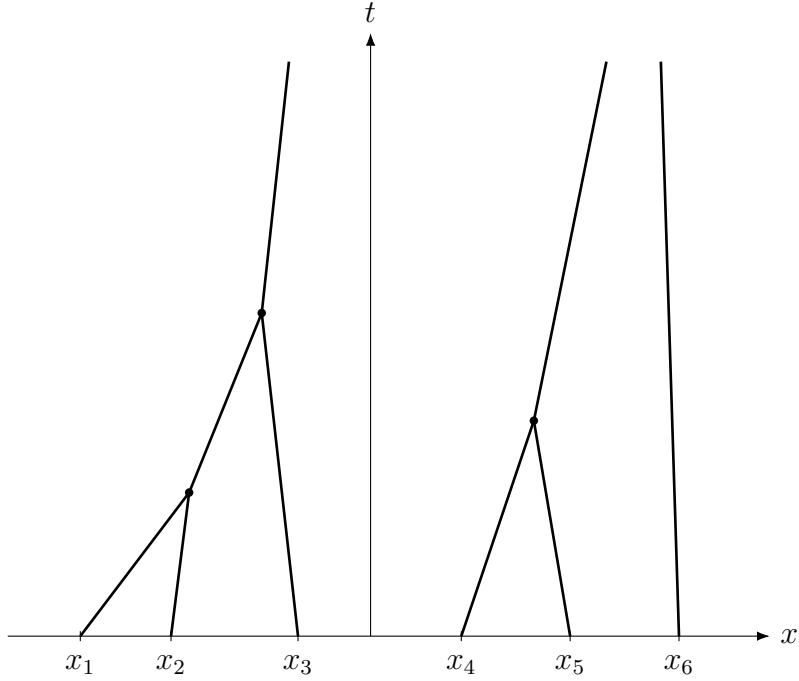
\begin{figure}[ht]
\centering
\begin{tikzpicture}[x=1.2cm,y=0.95cm,>=Latex]

\def\Tmax{8}

\draw[->] (-4.0,0) -- (4.4,0) node[right] {$x$};
\draw[->] (0,0) -- (0,\Tmax+0.4) node[above] {$t$};

\foreach \x/\lab in {-3.2/1, -2.2/2, -0.8/3, 1.0/4, 2.2/5, 3.4/6}{
    \draw[line width=0.4pt] (\x,-0.07) -- (\x,0.07);
    \node[below=3pt] at (\x,0) {$x_{\lab}$};
}

\draw[line width=1pt] (-3.2,0) -- (-2.0,2.0);
\draw[line width=1pt] (-2.2,0) -- (-2.0,2.0);
\fill (-2.0,2.0) circle (1.6pt);
\draw[line width=1pt] (-2.0,2.0) -- (-1.2,4.5);
\draw[line width=1pt] (-0.8,0) -- (-1.2,4.5);
\fill (-1.2,4.5) circle (1.6pt);
\draw[line width=1pt] (-1.2,4.5) -- (-0.9,\Tmax);

\draw[line width=1pt] (1.0,0) -- (1.8,3.0);
\draw[line width=1pt] (2.2,0) -- (1.8,3.0);
\fill (1.8,3.0) circle (1.6pt);
\draw[line width=1pt] (1.8,3.0) -- (2.6,\Tmax);

\draw[line width=1pt] (3.4,0) -- (3.2,\Tmax);

\end{tikzpicture}
\caption{Evolution of finitely many particles located at $x_1,\dots, x_6$. In this schematic, $v_0$ isn't nondecreasing. The trajectories of the particles are piecewise linear. When trajectories meet, the corresponding particles undergo a perfectly inelastic collision and continue together.  This is an example of a weak solution of the pressureless Euler equations described in \eqref{eq:sticky-soln-ex}.}
\label{fig:colliding-trajectories}
\end{figure}